\documentclass[12pt,a4paper]{amsart}
\usepackage[top=35mm, bottom=35mm, left=30mm, right=30mm]{geometry}
\usepackage{mathptmx}
\usepackage{hyperref}
\usepackage{xcolor}

\theoremstyle{plain}
\newtheorem{thm}{Theorem}[section]
\newtheorem{lem}[thm]{Lemma}
\newtheorem{prop}[thm]{Proposition}
\newtheorem{cor}[thm]{Corollary}

\theoremstyle{definition}
\newtheorem{defn}[thm]{Definition}
\newtheorem{rem}[thm]{Remark}

\newcommand{\Z}{{\mathbb{Z}_+}}
\newcommand{\N}{\mathbb{N}}

\newcommand{\mZ}{\mathbb Z}
\newcommand{\ep}{\varepsilon}
\DeclareMathOperator{\diam}{diam}
\DeclareMathOperator{\supp}{supp}

\begin{document}
\title{On proximality with Banach density one}
\author[J. Li]{Jian Li}
\date{\today}
\address[J.~LI]{Department of Mathematics, Shantou University, Shantou, Guangdong 515063, P.R. China}
\email{lijian09@mail.ustc.edu.cn}
\author[S.~Tu]{Siming Tu}
\address[S.~Tu]{Department of Mathematics, University of Science and Technology of China,
Hefei, Anhui, 230026, P.R. China}
\email{tsming@mail.ustc.edu.cn}
\begin{abstract}
Let $(X,T)$ be a topological dynamical system.
A pair of points $(x,y)\in X^2$ is called Banach proximal if
for any $\ep>0$, the set $\{n\in\Z:\ d(T^nx,T^ny)<\ep\}$ has Banach density one.
We study the structure of the Banach proximal relation. An useful tool is the
notion of the support of a topological dynamical system. We show that
a dynamical system is strongly proximal if and only if every pair in $X^2$ is Banach proximal.
A subset $S$ of $X$ is Banach scrambled if every two distinct points in $S$ form a Banach proximal pair
but not asymptotic. We construct a dynamical system with the whole space being a Banach scrambled set.
Even though the Banach proximal relation of the full shift is of  first category,
it has a dense Mycielski invariant Banach scrambled set.
We also show that for an interval map it is Li-Yorke chaotic if and only if
it has a Cantor Banach scrambled set.
\end{abstract}
\maketitle

\section{Introduction}
Throughout this paper, a {\em topological dynamical system} is a pair $(X, T)$,
where $X$ is a non-empty compact metric space with a metric $d$ and
$T$ is a continuous map from $X$ to itself.

A pair of points $(x,y)\in X^2$ is called \emph{proximal} if $\liminf_{n\to\infty}d(T^nx,T^ny)=0$.
The \emph{proximal relation} of $(X,T)$, denoted by $P(X,T)$, is the collection of all proximal pairs in $(X,T)$.
The study of proximal relation plays a big role in topological dynamics.
It is well known that the smallest $T\times T$-invariant closed equivalent relation generated by $P(X,T)$
is the maximal distal factor of $(X,T)$ (see \cite{A} for example).
But in general $P(X,T)$ is neither closed nor equivalent.

A subset $F$ of $\Z$ is said to be \emph{syndetic} if there exists $N\geq 1$ such that
$\{n,n+1,\dotsc,n+N\} \cap F\neq\emptyset$ for every $n\in \Z$.
The author in~\cite{C63} introduced the  notion of syndetically proximal relation of $(X,T)$.
A pair of points $(x,y)\in X^2$ is called \emph{syndetically proximal} if
for every $\ep>0$, $\{n\in\Z:\ d(T^nx,T^ny)<\ep\}$ is syndetic.
Denote by $SP(X,T)$  the collection of all syndetically proximal pairs in $(X,T)$.
It is shown in~\cite{C63} that $SP(X,T)$ is an invariant equivalence relation in $X$, but it may be not closed.
If $P(X,T)$ is closed in $X \times X$, then $P(X,T) = SP(X,T)$ and $P(X,T)$ is an invariant closed equivalence
relation in $X$.
Recently, syndetically proximal pairs were studied in~\cite{M11} and~\cite{MP13} in details.

A subset $F$ of $\Z$ is said to have \emph{Banach density one}
if for every $\lambda<1$ there exists $N\geq 1$ such that $\#(F\cap I)\geq \lambda \#(I)$
for every subinterval $I$ of $\Z$ with $\#(I)\geq N$, where $\#(I)$ denotes the number of elements of $I$.
A pair of points $(x,y)\in X^2$ is called \emph{Banach proximal} if
for every $\ep>0$, $\{n\in\Z:\ d(T^nx,T^ny)<\ep\}$ has Banach density one.
Clearly, every set with Banach density one is syndetic, then a Banach proximal pair is syndetically proximal.
Note that it is shown in~\cite{Li11} that for an interval map with zero topological entropy,
every proximal pair is Banach proximal.

This paper is devoted to the study of Banach proximal pairs.
The paper is organized as follows.
In Section 2, we recall  various types of density of subsets of non-negative integers,
some basic definitions in topological dynamics and show some properties of Banach proximal pairs.
In Section 3, we study the support of a dynamical system and show its connection
with the subsets of non-negative integers with some kind of density.
This property is very useful in the discussion of the structure of the Banach proximal relation.
In Section 4, we study the structure of the Banach proximal relation.
We show that a pair of points is Banach proximal if and only if
the support of the closure of the orbit of the pair in the product system is contained in the diagonal.
We also obtain some sufficient conditions for the Banach proximal cell of each point to be small.
In Section 5, we study the proximality on the induced spaces.
It is shown that a dynamical system is strongly proximal if and only if every pair in $X^2$ is Banach proximal.
In the finial section, we investigate Banach scrambled sets.
We construct a dynamical system with the whole space being a Banach scrambled set.
Even though the Banach proximal relation of the full shift is of  first category,
it has a dense Mycielski invariant Banach scrambled set.
We also show that for an interval map, it is Li-Yorke chaotic if and only if
it has a Cantor Banach scrambled set.

\section{Preliminaries}
In this section we recall some notions and aspects of the theory of topological dynamical
systems. We also show some properties of Banach proximal pairs.
\subsection{The density of subsets of non-negative integers}
Denote by $\Z$ ($\N$, $\mZ$, respectively) the set of all non-negative integers (positive integers, integers, respectively).
Let $F$ be a subset of $\mathbb{Z}_+$. Define the \emph{upper density} $\overline{d}(F)$ of $F$ by
\[ \overline{d}(F)=\limsup_{n\to\infty} \frac{\#(F\cap[0,n-1])}{n},\]
where $\#(\cdot)$ is the number of elements of a set.
Similar, $\underline{d}(F)$, the \emph{lower density} of $F$, is defined by
\[ \underline{d}(F)=\liminf_{n\to\infty} \frac{\#(F\cap[0,n-1])}{n}.\]
One may say $F$  has \emph{density} $d(F)$ if $\overline{d}(F)=\underline{d}(F)$,
in which case $d(F)$ is equal to this common value.
The \emph{upper Banach density} $BD^*(F)$ is defined by
\[BD^*(F)=\limsup_{N-M\to\infty} \frac{\#(F\cap[M,N])}{N-M+1}\]
Similarly, we can define the \emph{lower Banach density} $BD_*(F)$ and \emph{Banach density} $BD(F)$.
By the definition of Banach density,
it is easy to see that a subset $F$ of $\Z$ has Banach density one if and only if for every $\lambda<1$ there exists $N\geq 1$
such that $\#(F\cap I)\geq \lambda\#(I)$ for every subinterval of $\Z$ with $\#(I)\geq N$.
It is not difficult to observe the following.
\begin{lem}\label{lem:BD1}
Let $F,F_1,F_2$ be subsets of $\Z$.
\begin{enumerate}
  \item If  $F$ has Banach density one, then for every $n\in\mZ$,
  $F+n=\{k+n\in\Z:\ k\in F\}$ also has Banach density one.
  \item If  $F$ has Banach density one, then $\Z\backslash F$ is a set of Banach density zero.
  \item If $F_1$ and $F_2$ have Banach density one, then so does $F_1\cap F_2$.
  \item If $F_1$ has Banach density one and $F_2$ has positive upper Banach density,
  then $F_1\cap F_2$ has positive upper Banach density.
\end{enumerate}
\end{lem}

\subsection{Topological dynamics}
Let $(X,T)$ be a dynamical system.
A non-empty closed invariant subset $Y \subset X$ (i.e.,
$T Y \subset Y$ ) defines naturally a subsystem $(Y , T )$ of $(X, T)$.
A system $(X,T)$ is called {\em minimal} if it contains no proper subsystem.
Each point belonging to some minimal subsystem of $(X,T)$ is called a {\em minimal point}.

The \emph{orbit} of a point $x\in X$ is the set $Orb(x,T)=\{T^nx:\ n\in\Z\}$.
The set of limit points of the orbit $Orb(x,T)$ is called the \emph{$\omega$-limit set}
of $x$, and is denoted by $\omega(x,T)$.
A point $x\in X$ is a \emph{fixed point} if $Tx=x$;
\emph{periodic with least period $n$} if $n$ is the smallest positive integer satisfying $T^n(x) = x$;
\emph{recurrent} if $x\in \omega(x,T)$.

For $x\in X$ and $U,V\subset X$, put $N(x,U)=\{n\in\Z: T^nx\in U\}$ and
$N(U,V)=\{n\in\Z: U\cap T^{-n}V\neq\emptyset\}$.
Recall that a dynamical system $(X,T)$ is called {\em topologically transitive}
(or just {\em transitive}) if for every two non-empty open subsets $U,V$ of $X$
the set $N(U,V)$ is infinite.  Any point with dense orbit is called a \emph{transitive point}.
Denote the set of all transitive points by $Trans(X,T)$.
It is well known that for a transitive system, $Trans(X,T)$ is a dense $G_\delta$ subset of $X$.

A dynamical system $(X,T)$ is called {\em weakly mixing}
if the product system $(X^2, T\times T)$ is transitive;
and {\em strongly mixing} if for every two non-empty open subsets $U,V$ of $X$,
the set $N(U,V)$ is cofinite, i.e., there exists some $N\in \N$
such that $N(U,V)\supset \{N, N+1,\ldots\}$.

Let $(X,T)$ and $(Y,S)$ be two dynamical systems. If there is a continuous surjection
$\pi: X \to Y$ with $\pi\circ T = S\circ \pi$,
then we say that $\pi$ is a \emph{factor map}, the system
$(Y,S)$ is a \emph{factor} of $(X,T)$ or $(X, T)$ is an \emph{extension} of $(Y,S)$.
If $\pi$ is a homeomorphism, then we say that $\pi$ is a \emph{conjugacy} and
dynamical systems $(X,T)$ and $(Y,S)$ are \emph{conjugate}. Conjugate dynamical systems can
be considered the same from the dynamical point of view.

We refer the reader to the textbook~\cite{W82} for basic properties of topological entropy.
\subsection{Symbolic dynamics}

Consider $\{0,1\}$ as a topology space with the discrete topology
and let $\Sigma=\{0,1\}^{\Z}$ with the product topology.
Then $\Sigma$ is a Cantor space.
We write elements of $\Sigma$ as $x=x_0x_1x_2\dotsc$.
The \emph{shift map} $\sigma:\Sigma\to\Sigma$ is defined by
the condition that $\sigma(x)_n=x_{n+1}$ for $n\in\Z$.
It is clear that $\sigma$ is a continuous surjection.
The dynamical system $(\Sigma,\sigma)$ is called the \emph{full shift}.
If $X$ is non-empty, closed, and $\sigma$-invariant (i.e. $\sigma(X)\subset X$),
then the dynamical system $(X, \sigma)$ is called a \emph{subshift}.

Note that there is a natural order on $\{0,1\}$, that is $0\leq 0, 0\leq 1$ and $1\leq 1$.
We can extend this order coordinate-wise to a partial order $\leq $ on $\Sigma$,
that is $x\leq y$ if and only if $x_i\leq y_i$ for all $i\in\Z$.
A subshift $X\subset \Sigma$ is called \emph{hereditary}
provided that for any $x\in X$ if $y\leq x$ then $y\in X$ \cite{KL07,K13}.

Fix a subset $P$ of $\N$. Let $X_P$ be the subset of $\Sigma$ consisting of all
sequences $x$ such that if $x_i=x_j=1$ for some $i,j$ then
$|i-j|\in P\cup\{0\}$.
It is easy to see that $X_P$ is a subshift.
We call $(X_p,\sigma)$ the \emph{spacing subshift} generated by $P$.
The class of spacing shifts was introduced by Lau and Zame in~\cite{LZ72},
and for a detailed exposition of their properties we refer to~\cite{BNOST13}.
It is clear that every spacing subshift is hereditary.

A subset $P$ of $\N$ is said to
be \emph{thick} if it contains arbitrarily long intervals, that is for every $n\in\N$
there exists $a_n\in\N$ such that $\{a_n,a_n+1,\dots,a_n+n\}\subset P$.
It is shown in \cite{LZ72} that a spacing subshift $(X_P,\sigma)$ is weakly mixing if and only if $P$ is a thick set.

\subsection{Banach proximal pairs}
A pair of points $(x,y)\in X^2$ is said to be \emph{Banach proximal}
if for any $\ep>0$, the set $\{n\in\Z:\ d(T^nx,T^ny)<\ep\}$ has Banach density one.
Denote by $BP(X,T)$ the collection of all Banach proximal pairs of $(X,T)$.

\begin{lem}
Let $(X,T)$ be a dynamical system. Then $BP(X,T)$ is an invariant equivalence relation in $X$.
\end{lem}
\begin{proof}
This follows directly from the definition and Lemma~\ref{lem:BD1}.
\end{proof}

\begin{lem}
Let $(X,T)$ be a dynamical system and $k\in\mathbb{N}$.
Then $BP(X,T)=BP(X,T^k)$.
\end{lem}
\begin{proof}
For every $\ep>0$, choose $\delta>0$ such that $d(a,b)<\delta$ implies $d(T^ia,T^ib)<\ep$ for $i=0,1,\dotsc,k-1$.

Suppose that $(x,y)\in BP(X,T)$.
Let $E_1=\{n\in\Z: d(T^{n}x,T^{n}y)<\delta\}$ and
$E_2=\{n\in\Z: d((T^{k})^nx, (T^{k})^ny)<\ep\}$.
For every $n\in E_1$, we have $\lceil \frac n k \rceil \in E_2$,
where $\lceil z \rceil=\min\,\{n\in\mathbb{Z}: n\ge z\}$.
Since $E_1$ has Banach density one, it is easy to verify that so does $E_2$.
Then $(x,y)\in BP(X,T^k)$.

Now Suppose that $(x,y)\in BP(X,T^k)$.
Let $F_1=\{n\in\Z: d( (T^{k})^nx, (T^{k})^ny)<\delta\}$ and
$F_2=\{n\in\Z: d(T^{n}x,T^{n}y)<\ep\}$.
For every $n\in F_1$, we have $\{kn,kn+1,\dotsc,kn+k-1\}\in F_2$.
Since $F_1$ has Banach density one, it is easy to verify that so does $F_2$.
Then $(x,y)\in BP(X,T)$.
\end{proof}
\begin{lem}
Let $(X,T)$ be a dynamical system. Then $BP(X,T)$ is an $F_{\sigma\delta}$ subset of $X^2$.
\end{lem}
\begin{proof}
For every $k\in\N$, let $\mathfrak{A}_k$ be the collection of subinterval of $\Z$ with cardinality greater than $k$.
It is easy to verify that
\begin{align*}
BP(X,T)=\bigcap_{m=2}^\infty
\bigcup_{N=1}^\infty \bigcap_{k=N}^\infty \bigcap_{I\in \mathfrak{A}_k}
\Big\{(x,y)\in X^2: \#\big(\{n\in I: d(T^n(x),T^n(y))\leq \tfrac 1m \}\big)\\
\geq (1-\tfrac 1m) \#(I)\Big\}.
\end{align*}
Then $BP(X,T)$ is $F_{\sigma\delta}$.
\end{proof}

\begin{cor}
Let $(X,T)$ be a dynamical system and $x\in X$.
Then the  Banach proximal cell of $x$, $BP(x)=\{y\in X:\ (x,y)$ is  Banach proximal$\}$,
is an $F_{\sigma\delta}$ subset of $X$.
\end{cor}

\section{The support of a dynamical system}

Let $X$ be a compact metric space, $C(X)$ be the Banach algebra of all real valued continuous function on $X$
with the supremum norm, and $C^*(X)$ be the dual space of $C(X)$.
Let $M(X)$ be the set of regular Borel probability measures on $X$.
We can regard $M(X)$ as a subset of $C^*(X)$.
With the weak$^*$-topology induced from $C^*(X)$, $M(X)$ is a compact metrizable convex space.

Let $(X,T)$ be a topological dynamical system.
The action of $T$ on $X$ induces an action of $T_M$ on
$M(X)$ in the following way: for $\mu\in M(X)$ we define $T_M\mu$ by
\[\int_X f(x)\ dT_M\mu(x)= \int_X f(Tx)d\mu(x), \quad \forall f\in C(X).\]
So $T_M$ is the adjoint of the composition linear operator $C_T$ on $C(X)$, that is, $C_Tf=f\circ T$.
Clearly, $T_M$ is a continuous and linear operator. Hence $(M(X),T_M)$ is also a topological dynamical system.
For a point $x\in X$ we denote the \emph{point mass} at $x$ by $\delta_x$.
We may regard $(X,T)$ as a subsystem of $(M(X),T_M)$ by identifying $x\in X$ with $\delta_x\in M(X)$.

We are interested in those members of $M(X)$ that are invariant measures for $T$.
Let $M(X,T)=\{\mu\in M(X): T_M\mu=\mu\}$. This set consists of all
$\mu\in M(X)$ making $T$ a measure-preserving transformation of $(X,\mathcal{B}(X),\mu)$,
where $\mathcal{B}(X)$ is the Borel $\sigma$-algebra of $X$.
By the well known Krylov-Bogolioubov Theorem, $M(X,T)$ is nonempty.
In fact, $M(X,T)$ is a convex compact subset of $M(X)$.

An invariant measure $\mu\in M(X,T)$ is \emph{ergodic} if the only Borel sets $B$ with $T^{-1}(B)=B$
satisfy $\mu(B)=0$ or $\mu(B)=1$.
Let $M^e(X,T)$ denote the set of  ergodic measures in $M(X,T)$.
Then $M^e(X,T)$ coincides with the set of  extreme points of $M(X,T)$.
By the Choquet representation theorem,
for each $\mu\in M(X,T)$, there is a unique measure $\tau$ on the Borel subsets of $M(X,T)$
such that $\tau(M^e(X,T))=1$ and $\mu=\int_{M^e(X,T)}md\tau(m)$.
We call this the \emph{ergodic decomposition} of $\mu$.
A dynamical system $(X,T)$ is called \emph{uniquely ergodic} if $M(X,T)$ is a singleton.

The \emph{support} of a measure $\mu\in M(X)$, denoted by $\supp(\mu)$,
is the smallest closed subset $C$ of $X$ such that $\mu(C)=1$. It is clear that
\begin{align*}
\supp(\mu)&=\{x\in X:\ \text{for every open neiborhood $U$ of $x$, }\mu(U)>0\}\\
&=X\setminus\bigcup\{U\subset X:\ U \text{ is open and }\mu(U)=0\}
\end{align*}
The \emph{support} of a dynamical system $(X,T)$, denoted by $\supp(X,T)$,
is the smallest closed subset $C$ of $X$ such that $\mu(C)=1$ for all $\mu\in M(X,T)$ (see~\cite{S}).
It is clear that
\[\supp(X,T)=\overline{\bigcup\{\supp(\mu):\ \mu\in M(X,T)\}}.\]
Note that the support of a dynamical system is also called the \emph{measure center} in~\cite{Z93}.
\begin{lem}\label{lem:supp-mu-XT}
Let $(X,T)$ be a dynamical system.
Then there exists an invariant measure $\mu\in M(X,T)$ such that $\supp(\mu)=\supp(X,T)$.
\end{lem}
\begin{proof}
By the compactness of $X$, we can choose an at most countable subset $\{x_n:\ n\in\Lambda\}$
of $\bigcup\{\supp(\mu):\ \mu\in M(X,T)\}$
which is dense in $\supp(X,T)$.
For every $n\in\Lambda$, there exists $\mu_n\in M(X,T)$ such that $x_n\in \supp(\mu_n)$.
Choose a sequence of positive real numbers $\{a_n:\ n\in\Lambda\}$ with $\sum_{n\in\Lambda}a_n=1$.
Let $\mu=\sum_{n\in\Lambda}a_n\mu_n$.
It is easy to verify that $\mu\in M(X,T)$ and $\{x_n:\ n\in\Lambda\}\subset\supp(\mu)$.
Then $\supp(\mu)=\supp(X,T)$.
\end{proof}

\begin{rem}
It should be noticed that the measure in Lemma~\ref{lem:supp-mu-XT} may not be ergodic.
For example, let $X=\{a,b\}$ and $T$ is the identity map on $X$.
A more interesting example constructed in~\cite{W71} shows that there exists
a transitive system $(X,T)$ with $\supp(X,T)=X$ but
for every ergodic measure $\mu\in M(X,T)$, $\supp(\mu)\neq\supp(X,T)$.
\end{rem}

\begin{lem} \label{lem:supp-ergodic}
Let $(X,T)$ be a dynamical system. Then
\[\supp(X,T)=\overline{\bigcup\{\supp(\mu):\ \mu\in M^e(X,T)\}}.\]
\end{lem}
\begin{proof}
Choose a $\mu\in M(X,T)$ such that $\supp(\mu)=\supp(X,T)$.
Let $x\in\supp(\mu)$ and $U$ be a neighborhood of $x$.
Then $\mu(U)>0$. By the ergodic decomposition of $\mu$,
there exists some ergodic measure $\nu\in M(X,T)$ such that $\nu(U)>0$.
Then $U\cap\supp(\nu)\neq\emptyset$. Therefore $x\in \overline{\bigcup\{\supp(\mu):\ \mu\in M^e(X,T)\}}$.
\end{proof}

\begin{defn}[\cite{F81}] 
Let $(X,T)$ be a dynamical system, $x\in X$ and $\mu\in M(X,T)$.
We say that $x$ is a \emph{generic point} for $\mu$ if
\[\lim_{n\to\infty}\frac{1}{n}\sum_{i=0}^{n-1}f(T^ix)=\int f d\mu\]
for every continuous function $f\in C(X)$.
\end{defn}

\begin{prop}[\cite{F81}] \label{prop:ergodic-generic} 
Let $(X,T)$ be a dynamical system.
If $\mu\in M(X,T)$ is an ergodic measure, then
almost every point of $X$ (with respect to $\mu$) is generic for $\mu$.
\end{prop}

\begin{lem}\label{lem:genric-rec}
Let $(X,T)$ be a dynamical system and $\mu\in M(X,T)$.
If a point $y\in \supp(\mu)$ is generic for $\mu$, then
$y$ is recurrent with positive lower density, that is
for every open neighborhood $U$ of $y$, $N(y,U)$ has positive lower density.
\end{lem}
\begin{proof}
Fix an open neighborhood $U$ of $y$.
Choose a continuous function $f\in C(X)$ such that $0\leq f(x)\leq 1$ for all $x\in X$, $f(y)=1$
and $f(x)=0$ for $x\in U^c$.
Then
\[\#(N(y,U)\cap[0,n-1]\})\geq \sum_{i=0}^{n-1}f(T^iy) \]
Since $y$ is a generic point, we have
\[\lim_{n\to\infty}\frac{1}{n}\sum_{i=0}^{n-1}f(T^iy)=\int f d\mu>0.\]
Then,
\[\liminf_{n\to\infty}\frac{\#(N(y,U)\cap[0,n-1])}{n}>0,\]
that is $N(y,U)$ has positive lower density.
\end{proof}

\begin{lem}\label{lem:pubd-measure}
Let $(X,T)$ be a dynamical system.
Then for every $x\in X$ and every subset $F$ of $\mathbb{Z}_+$ with positive upper Banach density,
there exists an invariant measure $\mu\in M(X,T)$ with $\mu\bigl(\overline{\{T^nx:\ n\in F\}}\bigr)>0$.
\end{lem}
\begin{proof}
We follow the idea in the proof of Lemma~3.7 in~\cite{F81}.
Let $Y=\overline{\{T^nx:\ n\in F\}}$.
Since $F$ has positive upper Banach density, so does $N(x,Y)$.
Then there exists two sequences $\{a_k\}$ and $\{b_k\}$ of positive integers
with $\lim_{k\to\infty}(b_k-a_k)=\infty$ such that
\[\lim_{k\to\infty} \frac{\#(N(x,Y)\cap[a_k,b_k-1])}{b_k-a_k}>0.\]
Now, set
\[\mu_k=\frac{1}{b_k-a_k}\sum_{i=a_k}^{b_k-1}\delta_{T^ix}.\]
Let $\mu=\lim_{i\to\infty}\mu_{k_i}$ be a limit point of $\{\mu_k\}$ in the weak$^*$-topology.
Clearly, $\mu$ is an invariant measure of $(X,T)$ and
\[\mu(Y)\geq \lim_{i\to\infty}\mu_{k_i}(Y)=
\lim_{i\to\infty}\frac{1}{b_{k_i}-a_{k_i}}\sum_{i=a_{k_i}}^{b_{k_i}-1}\delta_{T^ix}(Y)\]
\[=\lim_{i\to\infty}\frac{\#(N(x,Y)\cap[a_{k_i},b_{k_i}-1])}{b_{k_i}-a_{k_i}}> 0.\]
This ends the proof.
\end{proof}

\begin{prop} \label{prop:supp-dense-rec-pld}
Let $(X,T)$ be a dynamical system.
Then the support of $(X,T)$ is the closure of the set of recurrent points with positive lower density.
\end{prop}
\begin{proof}
We first show that $\supp(X,T)$ has dense recurrent points with positive lower density.
By Lemma~\ref{lem:supp-ergodic}, it suffices to show that for every ergodic invariant measure $\mu$,
$\supp(\mu)$ has dense recurrent points with positive lower density.
Let $G_\mu$ denote the collection of generic points with respect to $\mu$.
By Proposition~\ref{prop:ergodic-generic}, $\mu(G_\mu)=1$ and then $\mu(G_\mu\cap\supp(\mu))=1$.
By Lemma~\ref{lem:genric-rec}, every point in $G_\mu\cap\supp(\mu)$ is recurrent with positive lower density.
Clearly, $G_\mu\cap\supp(\mu)$ is dense in $\supp(\mu)$,
so $\supp(\mu)$ has dense recurrent points with positive lower density.

Now suppose that there exists a recurrent point $y\in X$ with positive lower density which is not in $\supp(X,T)$.
Then there exists an neighborhood $U$ of $y$
and an neighborhood $V$ of $\supp(X,T)$ such that $\overline{U}\cap\overline{V}=\emptyset$.
Since $N(y,U)$ has positive lower density, by Lemma~\ref{lem:pubd-measure} there exists
an invariant measure $\mu\in M(X,T)$ such that $\mu(\overline{U})>0$.
Then $\supp(\mu)\cap \overline{U}\neq\emptyset$. This contradicts to the fact $\supp(\mu)\subset \supp(X,T)\subset V$.
\end{proof}

\begin{prop} \label{prop:support-min-BD1}
Let $(X,T)$ be a dynamical system. Then the support of $(X,T)$ is the small closed subset $K$ of $X$ such that
for every $x\in X$ and every open neighborhood $U$ of $K$, $N(x,U)$ has Banach density one.
\end{prop}
\begin{proof}
We first show that for every $x\in X$ and every open neighborhood $U$ of $\supp(X,T)$, $N(x,U)$ has Banach density one.
Let $U$ be an open neighborhood $U$ of $\supp(X,T)$. If there exists some point $x_0\in X$ such
that $N(x_0,U)$ does not have Banach density one, then $N(x_0, U^c)$ has positive upper Banach density.
By Lemma~\ref{lem:pubd-measure}, there exists an invariant measure $\mu\in M(X,T)$ such that $\mu(U^c)>0$.
Then $\supp(\mu)\cap U^c\neq\emptyset$. This is a contradiction.

Next we shall show that the support of $(X,T)$ is the smallest closed subset satisfying the above property.
Suppose that there exists a closed proper subset $Y$ of $\supp(X,T)$ satisfying the requirement.
Pick a point $y\in \supp(X,T)\setminus Y$. Since $Y$ is closed, there exists an neighborhood $U$ of $y$
and  an neighborhood $V$ of $Y$
such that $\overline{U}\cap \overline{V}=\emptyset$.
Then there exists an ergodic invariant measure  $\mu$ such that $\mu(U)>0$.
Choose a generic point $z\in U\cap \supp(\mu)$ for $\mu$.
By Lemma~\ref{lem:genric-rec}, $N(z,U)$ has positive lower density.
Then the upper density of $N(z,V)$ is less than one.
This is a contradiction.
\end{proof}

Similarly to the proofs of Propositions~\ref{prop:supp-dense-rec-pld} and~\ref{prop:support-min-BD1},
we have the following two results.
\begin{prop}\label{prop:supp-pubd-r}
Let $(X,T)$ be a dynamical system.
Then the support of $(X,T)$ is the closure of the set of recurrent points with positive upper Banach density.
\end{prop}
\begin{prop}
Let $(X,T)$ be a dynamical system. Then the support of $(X,T)$ is the small closed subset $K$ of $X$ such that
for every $x\in X$ and every open neighborhood $U$ of $K$, $N(x,U)$ has upper density one.
\end{prop}

\begin{prop}
Let $(X,T)$ be a dynamical system and $\pi:(X,T)\to (Y,S)$ be a factor map. Then
\begin{enumerate}
  \item $\supp(X^2,T\times T)=\supp(X,T)\times\supp(X,T)$,
  \item $\supp(Y,S)=\pi(\supp(X,T))$.
\end{enumerate}
\end{prop}
\begin{proof}
(1) Let $\mu\in M(X,T)$ with $\supp(\mu)=\supp(X,T)$.
Clearly, $\mu\times\mu\in M(X^2,T\times T)$.
Then $\supp(X^2,T\times T)\supset \supp(\mu\times\mu)=\supp(X,T)\times\supp(X,T)$.
By Proposition~\ref{prop:supp-pubd-r}, the set of recurrent points of $(X^2,T\times T)$
with positive upper Banach density is dense in $\supp(X^2,T\times T)$.
Note if $(x,y)\in X^2$ is recurrent with positive upper Banach density in $(X^2,T\times T)$,
then $x$ and $y$ are recurrent with positive upper Banach density in $(X,T)$.
By Proposition~\ref{prop:supp-pubd-r} again, $(x,y)\in\supp(X,T)\times\supp(X,T)$.
Thus $\supp(X^2,T\times T)\subset\supp(X,T)\times\supp(X,T)$.

(2) This follows from Proposition~\ref{prop:supp-pubd-r} and the fact that
if $y\in Y$ is recurrent with positive upper Banach density then there exists $x\in X$ with $\pi(x)=y$
such that $x$ is also recurrent with positive upper Banach density (see Proposition 4.5 in~\cite{Li12}).
\end{proof}

Recall that a dynamical system $(X,T)$ is called an \emph{E-system} if it is transitive and $\supp(X,T)=X$.

\begin{prop}[\cite{HPY07}]
Let $(X,T)$ be a transitive system and $x$ be a transitive point.
Then $(X,T)$ is an E-system if and only if for every non-empty open subset $U$ of $X$,
$N(x,U)$ has positive upper Banach density.
\end{prop}

If fact, we have the following characterization of the support of the closure of a orbit.
\begin{prop}\label{prop:orbit-ubd}
Let $(X,T)$ be a dynamical system and $x\in X$.
Then a point $y\in X$ is in the support of $(\overline{Orb(x,T)},T)$
if and only if
for every open neighborhood $U$ of $y$, $N(x,U)$ has positive upper Banach density.
\end{prop}
\begin{proof}
Without loss of generality, we can require $\overline{Orb(x,T)}=X$. First assume that
 $y\in\supp(\overline{Orb(x,T)},T)$.
Let $U$ be a neighborhood of $y$.
Then there exists an ergodic invariant measure $\mu\in M(X,T)$ such that $\mu(U)>0$.
Choose a generic point $z\in U\cap \supp(\mu)$ for $\mu$.
Then by Lemma~\ref{lem:genric-rec}, $N(z,U)$ has positive lower density.
For a finite subset $W$ of $N(z,U)$,
by the continuity of $T$ there exists an open neighborhood $z$ of $V$ such that
$T^n(V)\subset U$ for every $n\in W$.
Since $z\in \overline{Orb(x,T)}$, there is $k\in\Z$ such that $T^kx\in V$.
This implies that $k+W\subset N(x,U)$, and then $N(x,U)$ has positive upper Banach density.

Now assume that for every open neighborhood $U$ of $y$, $N(x,U)$ has positive upper Banach density.
By Lemma~\ref{lem:pubd-measure}, there exists an invariant measure $\mu\in M(X,T)$ such that $\mu(\overline{U})>0$.
Then $\overline{U}\cap \supp(\mu)\neq\emptyset$, and thus $y\in \supp(X,T)$.
\end{proof}

\section{The structure of Banach proximal relation}
In this section, we study the structure of the Banach proximal relation.
First using the support of the orbit closure,
we have the following equivalent conditions of Banach proximal pairs.
\begin{prop}\label{prop:Banach-proximal-equi}
Let $(X,T)$ be a dynamical system, $(x,y)\in X^2$ and let $\Delta_X$ denote the diagonal of $X^2$.
Then the following conditions are equivalent:
\begin{enumerate}
  \item $(x,y)$ is Banach proximal;
  \item $\supp\bigl(\overline{Orb((x,y),T\times T)},T\times T\bigr)\subset\Delta_X$;
  \item $\overline{Orb((x,y),T\times T)}\subset BP(X,T)$.
\end{enumerate}
\end{prop}
\begin{proof}
(1) $\Rightarrow$ (2) By the definition of Banach proximality,
for every open neighborhood $U$ of $\Delta_X$, $\{n\in\Z:\ (T\times T)^n(x,y)\in U\}$ has Banach density one.
Then by Proposition~\ref{prop:orbit-ubd},
$\supp\bigl(\overline{Orb((x,y),T\times T)},T\times T\bigr) \subset \Delta_X$.

(2) $\Rightarrow$ (3) Let $(u,v)\in \overline{Orb((x,y),T\times T)}$. Then
 $\supp\bigl(\overline{Orb((u,v),T\times T)},T\times T\bigr)\subset
\supp\bigl(\overline{Orb((x,y),T\times T)},T\times T\bigr)\subset\Delta_X$.
By Proposition~\ref{prop:support-min-BD1} for every open neighborhood $U$ of $\Delta_X$,
we have $\{n\in\Z:\ (T\times T)^n(u,v)\in U\}$ has Banach density one.
Then $(u,v)$ is Banach proximal.

(3) $\Rightarrow$ (1) is obvious.
\end{proof}

\begin{prop} \label{prop:BP-supp}
Let $(X,T)$ be a dynamical system.
If $(x,y)\in X^2$ is Banach proximal,
then \[\supp\bigl(\overline{Orb(x,T)},T\bigr)=\supp\bigl(\overline{Orb(y,T)},T\bigr)\]and
\[\supp\bigl(\overline{Orb((x,y),T\times T)},T\times T\bigr)
=\bigl\{(z,z):\ z\in \supp\bigl(\overline{Orb(x,T)},T\bigr)\bigr\}.\]
\end{prop}
\begin{proof}
Let $z\in \supp\big(\overline{Orb(x,T)},T\big)$ and $U$ be an open neighborhood of $z$.
Then there exists $\ep>0$ such that $B(z,2\ep)\subset U$.
By~Proposition~\ref{prop:orbit-ubd}, $N(x,B(z,\ep))$ has positive upper Banach density.
Since $(x,y)$ is Banach proximal, the set $F=\{n\in\Z: d(T^nx,T^ny)<\ep\}$ has Banach density one.
Then $N(y,U)$ has positive upper Banach density since $N(x,B(z,\ep))\cap F\subset N(y,U)$.
Therefore $z\in \supp\big(\overline{Orb(y,T)},T\big)$ and
$\supp\bigl(\overline{Orb(x,T)},T\bigr)\subset \supp\bigl(\overline{Orb(y,T)},T\bigr)$.
By the symmetry of $x$ and $y$, we have $\supp\bigl(\overline{Orb(x,T)},T\bigr)=\supp\bigl(\overline{Orb(y,T)},T\bigr)$.

Fix $a\in \supp\bigl(\overline{Orb(x,T)},T\bigr)$.
For every open neighborhood $U$ of $a$, by the proof above we have that
$N((x,y),U\times U)=N(x,U)\cap N(y,U)$ has positive upper Banach density.
Then $(a,a)\in \supp\bigl(\overline{Orb((x,y),T\times T)},T\times T\bigr)$.
Now fix $(b,c)\in \supp\bigl(\overline{Orb((x,y),T\times T)},T\times T\bigr)$.
Clearly, $b\in \supp\bigl(\overline{Orb(x,T)},T\bigr)$ and $c\in \supp\bigl(\overline{Orb(y,T)},T\bigr)$.
By Proposition~\ref{prop:Banach-proximal-equi},
$(b,c)\in \Delta_X$, that is $b=c$. This ends the proof.
\end{proof}

The collection of fixed points in $(X,T)$ is denoted by $Fix(X,T)$.

\begin{cor}\label{cor:xTnx-BP}
Let $(X,T)$ be a dynamical system, $x\in X$ and $n\in\N$.
If $(x,T^nx)$ is Banach proximal, then $\supp(\overline{Orb(x,T)},T)=Fix(\overline{Orb(x,T)},T^n)$.
\end{cor}

\begin{proof}
If $u\in \supp(\overline{Orb(x,T)},T)$,
then $(u,T^n u)\in \supp(\overline{Orb((x,T^nx),T\times T)})$.
By Proposition~\ref{prop:BP-supp},
we have $T^nu=u$.
\end{proof}

\begin{cor}\label{cor:xTnx-BP-trans}
Let $(X,T)$ be a transitive system, $x\in Trans(X,T)$ and $n\in\N$.
If there exists $n\in\N$ such that $(x,T^nx)$ is Banach proximal, then $\supp(X,T)=Fix(X,T^n)$.
\end{cor}

\begin{proof}
This follows from Corollary~\ref{cor:xTnx-BP} since $\overline{Orb(x,T)}=X$.
\end{proof}

\begin{prop}
Let $(X,T)$ be a transitive system.
If $\supp(X,T)\neq Fix(X,T^n)$ for every $n\in\N$.
Then the interior of $BP(x)$ is empty  for every $x\in X$.
\end{prop}
\begin{proof}
Suppose that there exists $x\in X$ such that the interior of $BP(x)$ is not empty.
Let $U\subset BP(x)$ be a non-empty open subset.
Then there is a transitive point $y\in U$ and $n\in \N$ such that $T^ny\in U$.
Therefore, $(y,T^ny)$ is Banach proximal since $BP(X,T)$ is an equivalence relation.
By Corollary~\ref{cor:xTnx-BP-trans}, $\supp(X,T)=Fix(X,T^n)$. This is a contradiction.
\end{proof}

\begin{cor}
Let $(X,T)$ be an infinite E-system.
Then the interior of $BP(x)$ is empty  for every $x\in X$.
\end{cor}

Let $(X,T)$ be a dynamical system. We say that $T$ is \emph{semi-open} if for every non-empty
subset $U$ of $X$, the interior of $TU$ is not empty.
\begin{prop}\label{prop:BP-first-category}
Let $(X,T)$ be a transitive system with $\supp(X,T)\setminus Fix(X,T^n)\neq\emptyset$ for every $n\in\N$.
If $T$ is semi-open, then for every $x\in X$, $BP(x)$ is of first category in $X$.
\end{prop}

\begin{proof}
If $T$ is semi-open, then using the category version of the Poincar\'{e} recurrence theorem
introduced in \cite{M07}, we have that for every Borel set $A$ with second category,
there exists $n\geq 1$ such that $A\cap T^n A\neq\emptyset$.
If there exists a point $y\in X$ such that $BP(y)$ is of the second category.
Then $BP(y)\cap Trans(X,T)$ is also of the second category.
There exists a transitive point $z$  and $n\geq 1$ such that $z, T^nz\in BP(y)$.
Then $(z,T^nz)$ is Banach proximal.
By Corollary~\ref{cor:xTnx-BP-trans}, $\supp(X,T)=Fix(X,T^n)$, which is a contradiction.
\end{proof}

\begin{cor}
Let $(X,T)$ be a non-periodic minimal system. Then for every $x\in X$, $BP(x)$ is of first category in $X$.
\end{cor}
\begin{proof}
It follows from Proposition~\ref{prop:BP-first-category} and the fact that $T$ is semi-open (see~\cite{A}).
\end{proof}

Recall that a measure $\mu\in M(X)$ is called \emph{non-atomic} if $\mu(\{x\})=0$ for every $x\in X$.

\begin{prop}
Let $(X,T)$ be a dynamical system and $x\in X$.
Then for every non-atomic invariant measure $\mu\in M(X,T)$, $\mu(BP(x))=0$.
\end{prop}
\begin{proof}
It is clear that $\mu\times\mu$ is an invariant measure for $(X^2,T\times T)$.
Since $\mu$ is non-atomic, $\mu\times \mu(\Delta_X)=0$.
By the ergodic decomposition of $\mu\times \mu$,
there is a unique measure $\tau$ on the Borel subsets of $M(X^2,T\times T)$
such that $\tau(M^e(X^2,T\times T))=1$ and $\mu\times \mu=\int_{M^e(X^2,T\times T)}md\tau(m)$.
Since $\mu\times \mu(\Delta_X)=\int_{M^e(X^2,T\times T)}m(\Delta_X)d\tau(m)$.
Then $m(\Delta_X)=0$ for almost all $m\in M^e(X^2,T\times T)$.
By the pointwise ergodic theorem, $m(BP(x)\times BP(x))=0$
for every measure $\mu\in M^e(X^2,T\times T)$ with $m(\Delta_X)=0$,
which implies that $\mu\times \mu(BP(x)\times BP(x))=0$ and hence $\mu(BP(x))=0$.
\end{proof}

\section{Proximality on the induced spaces}
Recall that a dynamical system $(X,T)$ is called \emph{proximal} if any two points $x,y\in X$ are proximal.
We have the following characterization of proximal systems.
\begin{thm}[\cite{AK03,M11}]\label{thm:proximal-SP}
Let $(X,T)$ be a dynamical system.
Then the following conditions are equivalent:
\begin{enumerate}
  \item $(X,T)$ is proximal;
  \item it has a fixed point which is the unique minimal point of $(X,T)$;
  \item $SP(X,T)=X^2$.
\end{enumerate}
\end{thm}

Recall that a dynamical system $(X,T)$ is called \emph{strongly proximal}
if the induced system on the measure space $(M(X),T_M)$ is proximal~\cite{G76}.
If $(X,T)$ is strongly proximal, then it is proximal,
since $(X,T)$ can be regarded as a subsystem of the proximal system $(M(X),T_M)$.
Similarity to Theorem~\ref{thm:proximal-SP}, we have the following characterization of strongly proximal systems.

\begin{thm}\label{thm:strongly-proximal}
Let $(X,T)$ be a dynamical system.
Then the following conditions are equivalent:
\begin{enumerate}
  \item\label{enum-sp} $(X,T)$ is strongly proximal;
  \item\label{enum-p-ue} $(X,T)$ is proximal and unique ergodic;
  \item\label{enum-supp-s} $\supp(X,T)$ is a singleton;
  \item\label{enum-BP} $BP(X,T)=X^2$;
  \item\label{enum-ud1} for every pair $(x,y)\in X^2$ and every $\ep>0$,
  the set $\{n\in\Z:\ d(T^nx,T^ny)<\ep\}$ has upper density one.
\end{enumerate}
\end{thm}

Before proving Theorem~\ref{thm:strongly-proximal},
we need to recall some notions concerning convex sets.
Let $Q$ be a subset of a locally convex space $E$.
We write $ex(Q)$ for the set of \emph{extreme points} of $Q$.
The \emph{closed convex hull} of $X$ is the smallest closed convex set containing $X$,
we write $\overline{co}(X)$ for this set. We will use the following result.

\begin{lem}[\cite{A}] \label{lem:exQ} 
Suppose $Q$ is a compact convex subset of the locally convex space $E$ and $Z$ is a subset of $Q$ such that
$\overline{co}(Z)=Q$. Then $ex(Q)\subset \overline{Z}$.
\end{lem}

\begin{proof}[Proof of Theorem~\ref{thm:strongly-proximal}]
\eqref{enum-sp} $\Rightarrow$ \eqref{enum-p-ue}
Since $(X,T)$ is strongly proximal, $(X,T)$ is also proximal and then there exists a unique minimal point $x_0\in X$.
Since $(M(X),T_M)$ is proximal, $\delta_{x_0}$ is the unique minimal point in $(M(X),T_M)$, and then it is a fixed point.
Note that $\mu\in M(X)$ is an invariant measure of $(X,T)$
if and only if $\mu$ is a fixed point in $(M(X),T_M)$.
Then $M(X,T)=\{\delta_{x_0}\}$ and $(X,T)$ is unique ergodic.

\eqref{enum-p-ue}$\Leftrightarrow$\eqref{enum-supp-s} is obvious.

\eqref{enum-supp-s} $\Rightarrow$ \eqref{enum-sp}
Assume that $\supp(X,T)=\{z\}$. Then $z$ is a fixed point in $(X,T)$ and $M(X,T)=\{\delta_z\}$.
Let $\mu \in M(X)$.
Note that any weak $*$ limit of $\{\frac{1}{n}\sum_{i=0}^{n-1}T^i_M\mu\}$ is
an invariant measure of $(X,T)$, then $\delta_z\in \overline{co}\bigl(\overline{Orb(\mu, T_M)}\bigr)$.
Since $\delta_z$ is an extreme point, by Lemma~\ref{lem:exQ} $\delta_z\in \overline{Orb(\mu, T_M)}$.
Then $\delta_z$ is the unique minimal point in $(M(X),T_M)$.

\eqref{enum-supp-s}  $\Rightarrow$ \eqref{enum-BP} Suppose that $BP(X,T)\neq X^2$, i.e.,
there exist a pair $(x,y)\in X^2\setminus \Delta_X$ such that $(x,y)\notin BP(X,T)$.
Then there is some  $\epsilon>0$ such that the set $F=\{n\in \Z_+:d(T^nx,T^ny)> \epsilon\}$
has positive upper Banach density.
By Lemma \ref{lem:pubd-measure} there exists an invariant measure $\mu_1\in M(X,T)$
with $\mu_1(\overline{\{T^nx:n\in F\}})>0$, and an invariant measure
$\mu\in M(X\times X,T\times T)$ with $\mu(\overline{\{(T^nx,T^ny):n\in F\}})>0$. Since
$\supp(X,T)$ is a singleton and $\supp(X\times X,T\times T)=\supp(X,T)\times \supp(X,T)$,
then we know that $\supp(X\times X,T\times T)\subset \Delta_X$. But $\overline{\{(T^nx,T^ny):n\in F\}}
\cap \Delta_X=\emptyset$ since $d(T^nx,T^ny)> \epsilon$ for $n\in F$, which is a contradiction.
\eqref{enum-BP}$\Rightarrow$\eqref{enum-ud1} is obvious.

\eqref{enum-ud1} $\Rightarrow$ \eqref{enum-supp-s}
Suppose that $\supp(X,T)$ is not a singleton.
Then $\supp(X^2,T\times T)\setminus\Delta_X\neq\emptyset$, since $\supp(X^2,T\times T)=\supp(X,T)\times \supp(X,T)$.
By Lemma~\ref{prop:supp-dense-rec-pld}  there is a pair $(u,v)\in \supp(X^2,T\times T)\setminus\Delta_X$
which are recurrent with positive lower density.
Choose $\ep>0$ such that $B(u,2\ep)\cap B(v,2\ep)=\emptyset$ and
$N(u,B(u,\ep))\cap N(v,B(v,\ep))$ has positive lower density,
Then the upper density of  $\{n\in\Z:\ d(T^nu,T^nv)<\ep\}$ is less than one.
\end{proof}

\begin{rem}\label{rem:hs-sp}
There are many examples of dynamical systems which are strongly proximal.
It is shown in~\cite{K13} that
a hereditary subshift has zero topological entropy if and only if it is proximal and unique ergodic,
and then if and only if it is strongly proximal.
\end{rem}

\begin{rem}
It is shown in~\cite{HZ02} that there exists a strongly mixing subshift which is strongly proximal~\cite{HZ02}.
Inspired by \cite[Theorem~6.5]{MP13b}, we can give a more simple example of this type.

Let $X\subset\{0,1\}^{\Z}$ be the collection of $x$ satisfying the following condition:
if $w$ is a word of length $2^n$ (where $n\in\N$) appearing in $x$ and if $w$ starts with $1$,
then $\#\{i:w_i\neq 0\}\leq n$.
It is clear that $(X,\sigma)$ is a subshift.
Let $L(X)$ denote the language of $X$, that is, the set $L(X):=\{x[0,k]:x\in X,k\geq 0\}$.
If $u,v\in L(X)$, then $u0^kw0^\infty\in X$ for all sufficiently large $k$, which shows that $(X,\sigma)$ is strongly mixing.
It is easy to see that every point $x\in X$ goes to $0^\infty$ with Banach density one.
Then the support of $(X,\sigma)$ is $\{0^\infty\}$,
which implies that $(X,\sigma)$ is strongly proximal.
\end{rem}

\begin{rem}
By the variational principle of topological entropy,
a strongly proximal system must have zero topological entropy.
But there are proximal systems with positive topological entropy~\cite{O10}.
There also exists a proximal system with uniformly positive entropy of all orders,
that is each of its finite covers by non-dense open subsets has positive topological entropy~\cite{HLY12}.
\end{rem}

\begin{prop}\label{prop:wm-BP}
Let $(X,T)$ be a weakly mixing system.
Then either $BP(X,T)=X^2$ (and $\supp(X,T)$ is a singleton) or $BP(X,T)$ is
a first category subset of $X\times X$ disjoint with $Trans(X^2,T\times T)$.
\end{prop}

\begin{proof}
If $\supp(X,T)$ is a singleton, then by Theorem~\ref{thm:strongly-proximal}, we have that $BP(X,T)=X^2$.
Now we assume that $\supp(X,T)$ is not a singleton.
Let $(x,y)$ be a transitive point of $(X^2,T\times T)$. 
Since $\supp(X^2,T\times T)=\supp(X,T)\times \supp(X,T)$,
we choose a point $(u,v)\in\supp(X^2,T\times T)\setminus \Delta_X$.
Let $U\times V$ be an open neighborhood of $(u,v)$ with $U\cap V=\emptyset$.
By Lemma~\ref{prop:orbit-ubd}, $N((x,y),U\times V)$ has positive upper Banach density.
Thus $(x,y)$ can not be Banach proximal.
\end{proof}

A dynamical system $(X,T)$ also induces a system on the hyperspace naturally~\cite{BS75}.
Let $K(X)$ be the hyperspace on $X$, i.e., the space of non-empty closed subsets of $X$ equipped with
the Hausdorff metric $d_H$ defined by
\[d_H(A,B)=\max\Bigl\{\max_{x\in A}\min_{y\in B} d(x,y),\ \max_{y\in B}\min_{x\in A}
d(x,y)\Bigr\}\ \text{for}\ A,B\in K(X).\]
The transformation $T$ induces natural a continuous self-map $T_K$ on the hyperspace $K(X)$ defined by
\[T_K(C)=TC,\ \text{for}\ C\in K(X).\]
Then $(K(X),T_K)$ is also a dynamical system.
It is natural to ask when $(K(X),T_K)$ is proximal.
Note that if a closed subset $Y$ of $X$ is strongly invariant (that is $TY=Y$),
then $Y$ is a fixed point in $(K(X),T_K)$.
If $(K(X),T_K)$ is proximal and $T$ is surjective, then $(X,T)$ must be trivial.
The following proposition shows that if $(K(X),T_K)$ is proximal then $(X,T)$ is ``almost'' trivial.
Recall that for $K\subset X$, if for every $\ep > 0$
there is $n\in\N$ with $\diam(T^n K)<\ep$, then we call $K$ a \emph{uniformly proximal set} ~\cite{AGHSY10}.
\begin{prop}
Let $(X,T)$ be a dynamical system. Then the following conditions are equivalent:
\begin{enumerate}
  \item $(K(X),T_K)$ is proximal;
  \item $\bigcap_{n=0}^\infty T^nX$ is a singleton;
  \item $X$ is a uniformly proximal set.
\end{enumerate}
\end{prop}
\begin{proof}
(1)$\Rightarrow$(2)
First note that we can regard $(X,T)$ as a subsystem of $(K(X),T_K)$ by identifying $x\in X$ with $\{x\}\in K(X)$.
By Theorem~\ref{thm:proximal-SP}, there exists a fixed point $x\in X$.
Let $Y=\bigcap_{n=0}^\infty T^nX$.
It is clear that $Y$ is closed and  strongly invariant.
Then $Y$ is a fixed point in $(K(X),T_K)$.
By Theorem~\ref{thm:proximal-SP} again, we have $Y=\{x\}$.

(2)$\Rightarrow$(3)
The sequence $\{T^n X\}$ is decreasing under the inclusion relation since $TX\subset X$.
By the compactness of  $X$, we have $\lim_{n\to\infty} \diam (T^n X)=0$.

(3)$\Rightarrow$(1)
There exists a sequence $\{k_n\}$ in $\N$ such that $\lim_{n\to\infty} \diam (T^{k_n} X)=0$.
By the compactness of $X$, there exists a point $x\in X$ such that
$\lim_{n\to\infty}d_H(T^{k_n} X,\{x\})=0$, that is $(X,\{x\})$ is proximal in $(K(X),T_K)$.
Therefore $(K(X),T_K)$ is proximal.
\end{proof}

\section{Banach scrambled sets}
Let $(X,T)$ be a dynamical system.
A pair of points $(x,y)\in X$ is called \emph{asymptotic} if $\lim_{n\to\infty}d(T^nx,T^ny)=0$.
A subset $S\subset X$ containing at least two points is called \emph{scrambled}
(\emph{syndetically scrambled}, \emph{Banach scrambled}, respectively)
if for any two distinct points $x,y\in S$, $(x,y)$ is proximal (syndetically proximal, Banach proximal, respectively)
but not asymptotic.
Note that a subset $S\subset X$ is scrambled (syndetically scrambled, Banach scrambled, respectively) for $(X,T)$
if and only if so is for $(X,T^k)$ for every $k\in\N$.

There are examples of  dynamical systems with the whole space being a scrambled set (see~\cite{HY01} and~\cite{HY02}).
For these examples, by Theorem~\ref{thm:proximal-SP} the whole space is also a syndetically scrambled set.
We show that the whole space also can be a Banach scrambled set.
Before the construction we need some  concepts . Recall that a dynamical system $(X,T)$ is \emph{scattering}
if for every minimal system $(Y,S)$ the product system $(X\times Y,T\times S)$ is transitive.
A point $x\in X$ is \emph{equicontinuous} if for every $\ep>0$, there is $\delta>0$
with the property that $d(x,y)<\delta$ implies $d(T^nx,T^ny)<\ep$
for every $y\in X$ and every $n\in\N$.
A transitive system $(X,T)$ is called \emph{almost equicontinuous} if there exists some equicontinuous point in $X$.
It is known that every almost equicontinuous transitive system $(X,T)$ is uniformly rigid, that is
for every $\ep>0$ there is $n\in\N$ such that $d(x,T^nx)<\ep$ for every $x\in X$.

\begin{thm}
There exists a dynamical system with the whole space being a Banach scrambled set.
\end{thm}
\begin{proof}
By~\cite[Corollary 4.14]{AG01}, there exists an almost equicontinuous scattering system $(X,T)$ which is not minimal.
Let $Y=\supp(X,T)$. Then $Y$ is a proper subset of $X$, since every almost equicontinuous E-system is minimal~\cite{GW93}.
Let $R=Y\times Y\cup \Delta_X$. Then $R$ is an invariant closed equivalence relation in $X$.
We collapse $Y$ to a single point, i.e., we take the quotient space $\tilde X=X/R$
and denote by $\tilde{T}$ the map induced by $T$ on $X$.
Then $(\tilde{X},\tilde{T})$ is a factor of $(X,T)$.
Then $(\tilde{X},\tilde{T})$ is strongly proximal, since the support of $(\tilde{X},\tilde{T})$ is a singleton.
Note that $(X,T)$ is uniformly rigid, so is $(\tilde{X},\tilde{T})$.
Every non-diagonal pair in $\tilde{X}^2$ is recurrence in $(\tilde{X}^2,\tilde{T}\times \tilde{T})$,
and then not asymptotic.
Hence, the whole space $\tilde{X}$ is a Banach scrambled set.
\end{proof}

We say that a subset $K$ of $X$ is a \emph{Mycielski set} if it can be presented as a countable union of Cantor sets.
For convenience, we restate here a version of Mycielski¡¯s theorem (\cite{M64}, Theorem 1) which we shall use later.

\begin{thm}[Mycielski Theorem]
Let $X$ be a complete second countable metric space without isolated points.
If $R$ is a dense $G_\delta$ subset of $X^2$,
then there exists a dense Mycielski subset $K$ of  $X$ such that $K^2\subset R\bigcup\Delta_X$.
\end{thm}

A dynamical system $(X,T)$ is said to be \emph{Li-Yorke chaotic} if it has an uncountable scrambled set.
Using the Mycielski Theorem, it is shown that many kinds of dynamical systems are Li-Yorke chaotic,
including weakly mixing systems~\cite{I89}, systems with positive topological entropy~\cite{BGKM02},
and infinite transitive systems with at least one periodic point~\cite{HY02}.
But in general, the Banach proximal relation is $F_{\sigma\delta}$ and may not $G_\delta$,
so we can not apply the Mycielski Theorem to show the existence of Banach scrambled sets directly.
For example, by Proposition~\ref{prop:wm-BP}
the Banach proximal relation of the full shift $(\Sigma,\sigma)$ is a first category subset of $\Sigma\times\Sigma$.
However, inspired by Proposition~4.14 in~\cite{MP13}, we will show in Corollary~\ref{cor:full-shift-BP}
that there is a dense Mycielski,
$\sigma$-invariant Banach scrambled subset
$S$ of $\Sigma$ for $\sigma$.

\begin{thm}\label{thm:SP-transitive-BS}
Let $(X,T)$ be transitive system with $X$ being infinite.
If $(X,T)$ is strongly proximal, then
there is a dense Mycielski, $T$-invariant Banach scrambled subset $S$ of $X$.
\end{thm}
\begin{proof}
For $m,n\in\Z$, put
\[R_{m,n}=\{(x,y)\in X^2: (T^mx,T^ny)\text{ is a recurrent point in }(X^2,T\times T)\}\]
Note that
\[R_{m,n}=\bigcap_{j=1}^\infty \bigcap_{N=1}^\infty\bigcup_{i=N}^\infty
\bigl\{(x,y)\in X^2: d(T^{i+m} x , T^m x)<\tfrac 1 j, d(T^{i+n} y ,T^n y) <\tfrac 1j\bigr\},\]
which implies that $R_{m,n}$ is a $G_\delta$ subset of $X^2$.
Choose a transitive point $z\in X$.
Then $R_{m,n}$ is dense in $X^2$ since $\{(T^iz,T^jz): i,j\in\Z\}\subset R_{m,n}$ for every $m,n\in\Z$.
Let 
\[R=Trans(X,T)\times Trans(X,T) \cap \bigcap_{m,n\in\Z}R_{m,n}.\]
Then $R$ is a dense $G_\delta$ subset of $X^2$.
By the Mycielski Theorem, there exists a dense Mycielski subset $S$ of $X$ such that $S\times S\subset R\bigcup \Delta_X$.
Let $K=\bigcup_{i=0}^\infty T^iS$.
Since $S$ is scrambled, $T|_S$ is injective,
and then $K$ is also a dense Mycielski subset $S$ of $X$.
For two distinct point $a,b\in K$, there exist $x,y\in S$ and $m,n\in\Z$ such that
$T^mx=a$ and $T^ny=n$.
If $x\neq y$, then $(a,b)$ is a recurrent point in $(X^2,T\times T)$ since $(x,y)\in R_{m,n}$.
If $x=y$, then $(a,b)$ is a recurrent point in $(X^2,T\times T)$ since $x\in Trans(X,T)$.
By the strong proximality of $(X,T)$, $(a,b)$ is Banach proximal, and then it is Banach scrambled.
\end{proof}

\begin{thm}\label{thm:sp-dmbcs}
Let $(X,T)$ be a dynamical system with $X$ being infinite.
Assume that there exists a sequence of strongly proximal transitive subsystems $(X_n,T)$ of $(X,T)$
such that $\overline{\bigcup_{n=1}^\infty X_n}=X$ and $X_i\cap X_j\neq\emptyset$ for $i,j\in\N$.
Then there exists a dense Mycielski, $T$-invariant Banach scrambled subset $S$ of $X$.
\end{thm}
\begin{proof}
Since $(X_n,T)$ is strongly proximal,
there exists a fixed point in $X_n$ which is the unique minimal point in $(X_n,T)$.
For $i\neq j$, $(X_i\cap X_j,T)$ is a subsystem of $(X,T)$ and then it contains some minimal point.
Then there is a fixed point $p\in X$ such that $p\in \bigcap_{n=1}^\infty X_n$.
Since a strongly proximal transitive system is either perfect or singleton.
We can choose a non-empty subset $\mathbb{I}$ of $\N$
such that  $\overline{\bigcup_{i\in\mathbb{I}} X_i}=X$,
$X_n$ is perfect and $X_i\neq X_j$ for every $n,i,j\in \mathbb{I}$ with $i\neq j$.
By Theorem~\ref{thm:SP-transitive-BS},
for every $i\in\mathbb{I}$ there is a dense Mycielski, $T$-invariant Banach scrambled subset $S_i$ of $X_i$.
Let $S=\bigcup_{i\in\mathbb{I}}S_i$.
Then $S$ is a dense Mycielski, $T$-invariant subset of $X$.
Fix any two distinct point $x,y\in S$.
There are $i,j\in\mathbb{I}$ such that $x\in S_i$ and $y\in S_j$.
If $i=j$, then $(x,y)$ is Banach proximal but not asymptotic.
Now assume that $i\neq j$.
Since every $(X_n,T)$ is strongly proximal, both $(x,p)$ and $(y,p)$ is Banach proximal,
and then $(x,y)$ is also Banach proximal.
By the proof of Theorem~\ref{thm:SP-transitive-BS},
we also have $x\in Trans(X_i,T)$ and $y\in Trans(X_j,T)$,
which implies that $(x,y)$ can not asymptotic since $X_i\neq X_j$.
Hence $S$ is Banach scrambled.
\end{proof}

\begin{thm}
Let $(X_P,\sigma)$ be a weakly mixing spacing shift.
Then there exists a dense Mycielski, $\sigma$-invariant Banach scrambled subset $S$ of $X_P$.
\end{thm}
\begin{proof}
Since $P$ is thick, there exists a thick set $Q\subset P$ such that
$\N\setminus Q$ is also thick.
For $n\in\N$, let $Q_n=Q\cup([0,n]\cap P)$.
Clearly, $\bigcup_{n=1}Q_n=P$. Then $\overline{\bigcap_{n=1}^\infty X_{Q_n}}=P$.
By \cite[Proposition 2.1]{LZ72}  and \cite[Theorem 3.6]{BNOST13}
$(X_{Q_n},\sigma)$ is weakly mixing and has zero topological entropy.
By Remark~\ref{rem:hs-sp}, $(X_{Q_n},\sigma)$ is also strongly proximal.
Now assumptions of Theorem~\ref{thm:sp-dmbcs} are satisfied, which ends the
proof.
\end{proof}

\begin{cor}\label{cor:full-shift-BP}
For the full shift $(\Sigma,\sigma)$,
there is a dense Mycielski, $\sigma$-invariant Banach scrambled subset $S$ of $\Sigma$ for $\sigma$.
\end{cor}

Now we consider the interval maps. By an interval map we mean a continuous map $f:[0, 1] \to [0, 1]$.
It is shown in~\cite{D05} that for an interval map, it has positive topological entropy if and only if
it has uncountable scrambled sets invariant under some power of $f$,
and such scrambled sets can be chosen to be syndetically scrambled~\cite{M11}.

\begin{thm}\label{thm:invertal-pe}
If an interval map $f$ has positive topological entropy, then there is a Cantor set $S\subset[0,1]$ such that
$S$ is Banach scrambled and  $f^{2^n}(S)\subset S$ for some $n\in\mathbb{N}$.
\end{thm}
\begin{proof}
Under the assumption, by Theorem~9 in~\cite{M11} there exists $n\in\N$, an $f^{2^n}$-invariant closed subset $X\subset [0,1]$
such that $(X,f^{2^n})$ is conjugate to the full shift $(\Sigma,\sigma)$.
Now the result follows from Corollary~\ref{cor:full-shift-BP}.
\end{proof}

\begin{thm}
If an interval map $f$ is Li-Yorke chaotic, then it has a Cantor, Banach scrambled set.
\end{thm}
\begin{proof}
If $f$ has positive topological entropy, the result follows from Theorem~\ref{thm:invertal-pe}.
If $f$ has zero topological entropy,  the result follows from \cite[Theorem~4.14]{Li11}
and the fact that
every proximal pair is Banach proximal for zero topological entropy maps~\cite{Li11}.
\end{proof}

\section*{Acknowledgement}
The first author is supported in part by STU Scientific Research Foundation for Talents (NTF12021),
Guangdong Natural Science Foundation (S2013040014084)
and National Natural Science Foundation of China (Grant No. 11326135). The second author is also
supported by National Natural Science Foundation of China (Grant Nos. 11071231, 11171320, 11371339).
The authors would like to thank Xiangdong Ye and the anonymous referee for many helpful comments and suggestions.

\bibliographystyle{amsplain}

\begin{thebibliography}{99}
\bibitem{AK03} E. Akin and S. Kolyada, \emph{Li-Yorke sensitivity}, Nonlinearity, \textbf{16} (2003) 1421-1433.

\bibitem{AG01} E. Akin, E. Glasner, \emph{Residual properties and almost equicontinuity},
 J. Anal. Math., \textbf{84} (2001), 243--286.

\bibitem{AGHSY10} E. Akin, E. Glanser, W. Huang, S. Shao and X. Ye, \emph{Sufficient conditions under which a
transitive system is chaotic}, Ergodic Theory Dynam. Systems., \textbf{40} (2010), 1277--1310.

\bibitem{A} J. Auslander, \emph{Minimal Flows and Their Extensions}, North-holland, 1988.

\bibitem{BNOST13} J. Banks, T. Nguyen, P. Oprocha, B. Stanley, and B. Trotta,
\emph{Dynamics of spacing shifts}, Discrete Contin. Dyn. Syst., 33(9) (2013), 4207--4232.

\bibitem{BS75} W. Bauer and K. Sigmund, \emph{Topological dynamics of transformations induced
on the space of probability measures}, Monatsh. Math., \textbf{79} (1975) 81--92.

\bibitem{BGKM02}F. Blanchard, E. Glasner, S. Kolyada, and A. Maass, \emph{On Li-Yorke pairs},
J. Reine Angew. Math., 547 (2002), 51--68.

\bibitem{C63} J.P. Clay, \emph{Proximity relations in transformation groups},
Trans. Amer. Math. Soc., \textbf{108}(1) (1963), 88--96.

\bibitem{D05} B. Du, \emph{On the invariance of Li-Yorke chaos of interval maps},
 J. Difference Equ. Appl., \textbf{11} (9) (2005) 823¨C828.

\bibitem{F81} H. Furstenberg, \emph{Recurrence in Ergodic Theory and Combinatorial Number Theory},
M.B. Porter Lectures, Princeton University Press, Princeton, NJ, 1981.

\bibitem{G76} S. Glasner, \emph{Proximal Flows}, Lecture Notes in Math., 517, Springer, Berlin, 1976.

\bibitem{GW93} E. Glasner and B. Weiss, \emph{Sensitive dependence on initial conditions},
Nonlinearity, \textbf{6} (1993), 1067--1075.

\bibitem{HZ02} W. He, Z. Zhou, \emph{A topologically mixing system with its measure center being a singleton},
Acta Math. Sinica, \textbf{45} (5) (2002) 929--934.

\bibitem{HLY12} W. Huang, H. Li and X. Ye, \emph{Family-independence for topological and measurable dynamics},
Trans. Amer. Math Soc., \textbf{364} (2012), 5209--5245.

\bibitem{HPY07} W. Huang, K.K. Park, X. Ye, \emph{Dynamical systems disjoint from all minimal systems with zero entropy},
  Bull. Soc. Math. France, \textbf{135} (2007) 259--282.

\bibitem{HY01} W. Huang, X. Ye, \emph{Homeomorphisms with the whole compacta being scrambled sets},
  Ergodic Theory Dynam. Systems, \textbf{21} (2001) 77--91.

\bibitem{HY02} W. Huang, X. Ye, \emph{Devaney¡¯s chaos or 2-scattering implies Li-Yorke chaos},
Topology Appl., \textbf{117} (2002) 259--272.

\bibitem{I89} A. Iwanik, \emph{Independence sets of transitive points}, in: Dynamical Systems and Ergodic Theory,
Banach Center Publications, vol. \textbf{23}, (1989), pp. 277--282.

\bibitem{KL07} D. Kerr, H. Li, \emph{Independence in topological and C$^*$-dynamics}, Math. Ann., \textbf{338}
(2007), no. 4, pp. 869¨C-926.

\bibitem{K13} D. Kwietniak, \emph{Topological entropy and distribuitonal chaos
in hereditary shifts with applications to spacing shifts and beta shifts}.
Discrete Contin. Dyn. Syst., \textbf{33} (2013), 2451--2467.

\bibitem{LZ72} K. Lau and A. Zame, \emph{On weak mixing of cascades}, Math. Systems Theory, \textbf{6} (1972/73), 307--311.

\bibitem{Li11} J. Li, \emph{Chaos and entropy for interval maps},
J. Dyn. Differ. Equ., \textbf{23} (2011), no. 2, 333--352.

\bibitem{Li12} J. Li, Dynamical characterization of C-sets and its application, Fund. Math., \textbf{216} (2012) 259--286.

\bibitem{M07} P. Malicky, \emph{Category version of the Poincar\'{e} recurrence theorem},
Topology Appl. \textbf{154} (14) (2007) 2709--2713.

\bibitem{M11} T.K.S. Moothathu, \emph{Syndetically proximal pairs}, J. Math. Anal. Appl., \textbf{379} (2011) 656--663.

\bibitem{MP13} T.K.S. Moothathu and P. Oprocha, \emph{Synetical proximality and scrambled sets},
Topol. Methods Nonlinear Anal., \textbf{41} (2013), 421--461.

\bibitem{MP13b} T.K.S. Moothathu and P. Oprocha, \emph{Shadowing, entropy and minimal subsystems},
Monatshefte f\"ur Mathematik, \textbf{172} (2013), 357--378.

\bibitem{M64} J. Mycielski, \emph{Independent sets in topological algebras}, Fund. Math., \textbf{55} (1964), 137--147.

\bibitem{O10} P. Oprocha, \emph{Families, filters and chaos}, Bull. London Math. Soc., \textbf{42} (4) (2010), 713--725.

\bibitem{S} K. Sigmund, \emph{On minimal centers of attraction and generic points},
J. Reine Angew. Math., \textbf{295} (1977), 72--79.

\bibitem{W82} P. Walters, \emph{An introduction to ergodic theory}, Grad. Texts in Math., Springer, 1982.

\bibitem{W71} B. Weiss, \emph{Topological transitivity and ergodic measures},
Math. systems theory, \textbf{5} (1971) 71--75.

\bibitem{Z93} Z. Zhou, \emph{Weakly almost periodic point and measure centre},
Science in China, Ser. A, \textbf{36} (1993), 142--153.

\end{thebibliography}

\end{document}